\documentclass[11pt,a4paper,reqno]{amsart}
\usepackage[left=30mm,top=37mm,right=30mm,bottom=37mm]{geometry}
\usepackage{epsfig}
\usepackage{amsfonts}
\usepackage{amsmath}
\usepackage{amssymb}
\usepackage{amsthm}
\usepackage{times}
\usepackage{color}
\usepackage{enumerate}
\usepackage[T1]{fontenc}
\usepackage{inputenc}
\usepackage{graphicx}
\usepackage{nicefrac}
\usepackage{xcolor}
\usepackage{mathtools}
\usepackage[all]{xy}
\usepackage{comment}
\usepackage{hyperref}
\hypersetup{
 colorlinks=true, 
 linkcolor=blue, 
 citecolor=blue, 
 filecolor=blue, 
 urlcolor=blue 
}
\usepackage[nameinlink,capitalize,noabbrev]{cleveref}
\usepackage{tikz-cd}

\usepackage{pgfplots}
\pgfplotsset{compat=1.15}
\usepackage{mathrsfs}
\usetikzlibrary{arrows}
\usepackage[normalem]{ulem}

\usepackage{amsmath,amscd}

\newtheorem{theorem}{Theorem}[section]
\newtheorem*{theorem*}{Theorem}

\newtheorem{lemma}[theorem]{Lemma}
\newtheorem{proposition}[theorem]{Proposition}

{\theoremstyle{remark}
 \newtheorem{remark}[theorem]{Remark}}
{\theoremstyle{definition}
 \newtheorem{definition}[theorem]{Definition}

 \newtheorem{example}[theorem]{Example}

}

\newtheorem{introthm}{Theorem}

\def\Q{\mathbb{Q}}

\def\Z{\mathbb{Z}}

\def\H{\mathcal{H}}

\def\KK{\mathbf{k}}

\newcommand{\ZZ}[0]{\ensuremath{\mathbb{Z}}}
\newcommand{\GA}[0]{\ensuremath{\mathbb{G}_{\mathrm{a}}}}
\newcommand{\GM}[0]{\ensuremath{\mathbb{G}_{\mathrm{m}}}}
\newcommand{\AF}[0]{\ensuremath{\mathbb{A}}}

\newcommand{\QQ}[0]{\ensuremath{\mathbb{Q}}}

\newcommand{\spec}[0]{\ensuremath{\operatorname{Spec}}}

\newcommand{\Aut}[0]{\ensuremath{\operatorname{Aut}}}

\newcommand{\cone}[0]{\ensuremath{\operatorname{cone}}}

\newcommand{\SL}[0]{\ensuremath{\operatorname{SL}}}
\newcommand{\slt}
[0]{\ensuremath{\operatorname{SL}}}

\newcommand{\id}[0]{\ensuremath{\operatorname{id}}}

\makeatletter \newcommand*\bigcdot{\mathpalette\bigcdot@{.5}} \newcommand*\bigcdot@[2]{\mathbin{\vcenter{\hbox{\scalebox{#2}{$\m@th#1\bullet$}}}}} \makeatother

\DeclareMathOperator{\GL}{GL}

\begin{document}

\title[On the characterization of affine toric varieties by their automorphism group]{On the characterization of affine toric varieties \\ by their automorphism group}

\author{Roberto D\'iaz}
\address{Departamento de Matem\'aticas, Facultad de Ciencias, Universidad de Chile, Las Palmeras 3425, \~{N}u\~{n}oa, Santiago, Chile.}%
\email{robertodiaz@uchile.cl}

\author{Alvaro Liendo} %
\address{Instituto de Matem\'atica y F\'\i sica, Universidad de Talca,
 Casilla 721, Talca, Chile.}%
\email{aliendo@inst-mat.utalca.cl}

	\author{Andriy Regeta}
	\address{\noindent Institut f\"{u}r Mathematik, Friedrich-Schiller-Universit\"{a}t Jena, Jena 07737, Germany}
\email{andriyregeta@gmail.com}

\date{\today}

\thanks{
 \mbox{\hspace{11pt}}{\it Key words}: toric varieties, automorphism group, locally nilpotent derivation.\\
 \mbox{\hspace{11pt}} The first author is supported by the Fondecyt postdoc project 3230406. The second author was partially supported by Fondecyt project 1200502. The third author is supported by DFG, project number 509752046.}

\begin{abstract}
In this paper we show that a normal affine toric variety $X$ different from the algebraic torus is uniquely determined by its automorphism group in the category of affine irreducible, not necessarily normal, algebraic varieties if and only if $X$ is isomorphic to the product of the affine line and another affine toric variety. In the case where $X$ is the algebraic torus $T$, we reach the same conclusion if we restrict the category to only include irreducible varieties of dimension at most the dimension of $T$. There are examples of varieties of dimension one higher than $T$ having the same automorphism group of $T$. Hence, this last result is optimal.
\end{abstract}

\maketitle


\section*{Introduction}

Let $\KK$ be an uncountable algebraically closed field of characteristic zero and let $\GM$ be the multiplicative group of $\KK$. An affine toric variety is an irreducible affine variety $X$ that contains an algebraic torus $T\simeq\GM^n$ as an open Zariski subset such that the action of $T$ on itself extends to an algebraic action of $T$ on $X$~\cite[Definition~1.1.3]{cox2011toric}. This definition does not assume the property of normality on $X$ unlike, for instance, the definitions in \cite{fulton1993introduction, oda1983convex}. Let $X_1$ and $X_2$ be toric varieties with acting tori $T_1$ and $T_2$, respectively. A toric morphism is a regular map $\varphi \colon X_1\to X_2$ such that $\varphi(T_1)\subset T_2$ and the restriction $T_1\to T_2$ is a group homomorphism.

Toric varieties were first introduced by Demazure \cite{demazure1970sous} as a tool to study the structure of the Cremona group. They admit a combinatorial structure that allows us to describe them in a rather simple way. For example, normal toric varieties are described by means of certain combinatorial gadgets called fans and affine toric varieties are described by affine semigroups \cite{cox2011toric}. The description of non-normal non-affine toric varieties is more involved and will not be used in this paper. 

By construction, toric varieties have a rather large automorphism group containing at least the algebraic torus. The automorphism group of a complete toric variety is linear algebraic \cite{demazure1970sous} and in most cases is reduced only to the algebraic torus; see also \cite{liendo2021automorphisms}. On the other hand, the automorphism group of a normal affine toric variety $X$ is never algebraic when the dimension of $X$ is greater than one. Moreover, such automorphism group is always infinite dimensional when $X$ is different from the algebraic torus. In this context, it is natural to ask ourselves whether a normal affine toric variety is uniquely determined by its automorphism group in the category of irreducible varieties. This paper provides a definitive answer to this question in \cref{maintheorem}. Our theorem builds on several previous results that we describe now.

The automorphism group of an affine variety is quite often infinite dimensional, so they do not have the structure of an algebraic group. Nevertheless, they carry the structure of an ind-group. An ind-variety is a set $\mathcal{V}$ endowed with a filtration $V_1\hookrightarrow V_2\hookrightarrow\dots$ such that $\mathcal{V}=\bigcup V_i$, where each $V_i$ is an algebraic variety and the inclusions
$\varphi_i\colon V_i{\hookrightarrow} V_{i+1}$ are closed embeddings. An ind-group is a group object in the category of ind-varieties, see \cite{furter2018geometry} for details. 

It is proved in \cite[Corollary D]{regeta2021characterizing} that a normal affine toric variety not isomorphic to the algebraic torus is determined by its automorphism group in the category of normal affine varieties, see also \cite[Theorem A]{cantat2019families} and \cite[Theorem 1.3 and Theorem 1.4]{LRU23}. For the algebraic torus $T$ such a statement fails: if $C$ is a smooth affine curve that has trivial group of automorphisms and no non-constant invertible regular functions, then $\Aut(T) \simeq \Aut(T \times C)$ as ind-groups, see
\cite[Example 6.17]{LRU23}. However, it is proved in \cite[Theorem~E]{regeta2021characterizing} that the algebraic torus is determined by its automorphism group among normal affine varieties of dimension less than or equal to $\dim T$.

In \cite[Theorem 1]{regeta2022characterizing} and \cite[Lemma 9.1]{regeta2021characterizing2} it is proved that for every complex affine $n$-dimensional toric variety $X$ endowed with a regular non-trivial $\slt_n$-action there are infinitely many affine non-normal toric varieties $Y$ such that $\Aut(X) \simeq \Aut(Y)$ as ind-groups provided that $X$ is different from the affine space. All such varieties $Y$ are also classified in \cite[Theorem 1]{regeta2022characterizing}. This shows that, in general, affine toric varieties are not determined by their automorphism groups in the category of irreducible, not necessarily normal, affine varieties. In addition, in \cite[Theorem 4.5]{DL24} it is shown that for every normal affine toric surface $X$ different from the algebraic torus, the affine space, and $\GM\times \mathbb{A}^1$, there exists a non-normal affine toric surface $Y$ with $\Aut(X)\simeq \Aut(Y)$.

In this paper, we complete the picture by generalizing the last result to higher dimension. In fact, we prove the following theorem.

\begin{introthm}\label{maintheorem}
Let $X$ be a normal affine toric variety, and let $Y$ be an affine irreducible, not necessarily normal, variety such that $\Aut(Y) \simeq \Aut(X)$. 
\begin{enumerate}
 \item \label{main1} If $X$ is the algebraic torus and $\dim Y \le \dim X$, then $Y \simeq X$;
	\item \label{main2} If
 $X$ is isomorphic to $\mathbb{A}^1 \times Z$, where $Z$ is an affine toric variety, then $Y \simeq X$;
 \item \label{main3} If $X$ is not isomorphic to neither the algebraic torus nor $\mathbb{A}^1 \times Z$ as in (2), then there are infinitely many non-normal affine varieties $Y$ such that $\Aut(X) \simeq \Aut(Y)$. Moreover, the groups $\Aut(X)$ and $\Aut(Y)$ are isomorphic as ind-groups.
 \end{enumerate}
\end{introthm}

\subsection*{Acknowledgments}
We would like to thank the anonymous referee of the paper for valuable suggestions. In particular, the current version of \cref{Lemma: subset distinguished ray}, that greatly improved the exposition, was hinted to us by the referee.

\section{Preliminaries} 

\subsection{Affine semigroups}
An affine semigroup $S$ is a finitely generated commutative monoid that admits an embedding into a free abelian group $M$ of rank $n\in \Z_{\geq 0}$ \cite[Page 16]{cox2011toric}. Fix such an embedding $S\hookrightarrow M$. An affine semigroup is called saturated if, for all $k$ in $\Z_{>0}$ and $m$ in $M$, $km \in S$ implies $m \in S$. The saturation of a semigroup $S$ is the smallest saturated semigroup of $M$ that contains $S$. We denote the saturation of $S$ by $\widetilde{S}$. A morphism $\varphi\colon S\to S'$ between affine semigroups $S$ and $S'$ satisfies $\varphi(m+m')={\varphi}(m)+{\varphi}(m')$ and ${\varphi}(0)=0$.

The category of affine toric varieties with toric morphisms is dual with the category of affine semigroups with semigroup morphisms. Moreover, the subcategory of normal affine toric varieties is dual with the category of saturated affine semigroups. The correspondence between the objects is obtained via the notion of semigroup algebra.

Letting $S$ be an affine semigroup, we let $\KK[S]$ be the semigroup algebra of $S$ defined as
$$\KK[S]=\bigoplus_{m\in S} \KK\chi^m\quad\mbox{where}\quad \chi^m\cdot\chi^{m'}=\chi^{m+m'}\mbox{ and } \chi^0=1\,.$$ 
Now for every affine semigroup $S$, we obtain an affine toric variety $X_S$ by setting $X_S=\spec \KK[S]$. In the case where $S=M$ is a free abelian group of rank $n$, this construction yields $X_S=T$, the algebraic torus of dimension $n$. Conversely, given an affine toric variety $X$ we obtain the corresponding semigroup $S$ by considering the embedding $\spec\KK[M]=T\hookrightarrow X$. The semigroup corresponding to $X$ is obtained as
$$S=\{m\in M\mid \chi^m\colon T\to \GM \mbox{ extends to a regular function on } X\}\,.$$
If $X_S$ is the non-normal affine toric variety associated with the non-saturated affine semigroup $S$, the normalization of $X_S$ is the toric variety $X_{\widetilde{S}}$ of the saturation $\widetilde{S}$ of $S$.

\subsection{Demazure roots} 
\label{section-demazure}

The combinatorial structure of the affine semigroup $S$ allows us to study certain regular actions of the additive group $\GA:=(\mathbf{k},+)$ on $X_S$ via the notion of Demazure roots introduced in \cite{demazure1970sous}, see also \cite{L10,ArLi12}. We present now a generalization of this notion adapted to non-saturated affine semigroups following \cite[Definition 3.1]{DL24} and we refer to the original paper for the proofs.

Let $S$ be an affine semigroup, a submonoid $F\subset S$ is said to be a face if $m+m'\in F$ with $m,m'\in S$ implies $m,m'\in F$. We say that $S$ is pointed if $\{0\}$ is a face of $S$. A ray of $S$ is a face $F\simeq \Z_{\geq 0}$. Every ray $\rho\subset S$ has a unique generator as semigroup. As is customary in toric geometry, we also denote this generator by $\rho$ and we freely identify the ray and the ray generator. The set of rays of $S$ is denoted by $S(1)$. Remark that we do not require that $\rho\in S(1)$ is a primitive element.

Let also $S$ be embedded in $M\simeq \Z^n$. We say that $M$ is minimal if $S$ cannot be embedded in any subgroup of $M$. In the sequel, we always assume that an affine semigroup is embedded in a minimal free abelian group $M$. We also let $N:=\operatorname{Hom}{(M,\Z)}\simeq \Z^n$. This yields a bilinear map $\langle\ ,\ \rangle \colon M\times N\rightarrow \Z$. We define the dual semigroup $S^*=\{u\in N\mid \langle m,u\rangle\geq 0, \forall m\in S\}$. Remark that by its definition, the dual semigroup $S^*$ is always saturated. A facet $F$ of $S$ is a face obtained as 
$$F=\rho^\bot\cap S=\{m\in S\mid \langle m,\rho\rangle=0\},\quad\mbox{for some}\quad  \rho\in S^*(1)\,.$$

We now come to the definition of the Demazure roots of an affine semigroup.

\begin{definition}[{\cite[Definition~3.1]{DL24}}]\label{definition-root} %
Let $S$ be an affine semigroup embedded minimally in a free abelian group $M$ of finite rank. An element $\alpha\in M$ is called a Demazure root of $S$ if 
\begin{enumerate}[$(i)$]
 \item There exists $\rho\in S^*(1)$ such that $\langle \alpha, \rho \rangle=-1$, and
 \item The element $m+\alpha$ belongs to $S$ for all $m\in S$ such that $\langle m,\rho \rangle >0$.
\end{enumerate}
 We say that $\rho$ is the distinguished ray of $\alpha$. We denote the set of Demazure roots of $S$ by $\mathcal{R}(S)$ and the set of Demazure roots of $S$ with the distinguished ray $\rho$ by $\mathcal{R}_{\rho}(S)$. 
\end{definition}

Remark that, even if a ray generator $\rho\in S^*(1)$ is not required to be primitive by the definition of ray, only a ray with primitive generator can fulfil condition $(i)$ in the above definition. 

Recall that for any variety $Z$ endowed with a $T$-action, a root subgroup $U\subset \Aut(Z)$ is a subgroup isomorphic to the additive group $\GA$ that is normalized by $T$. Let $X_S$ be an affine toric variety. The interest on Demazure roots comes from the fact that they are in one-to-one correspondence with root subgroups of $\Aut(X_S)$.  In the case of complete smooth toric varieties, this was shown in \cite{demazure1970sous}, see also \cite{liendo2021automorphisms}. This was later revisited for normal affine toric varieties in \cite{L10}. Finally, for non-normal affine toric varieties, this was proven in \cite{DL24}.

The correspondence is obtained as follows. Let $S$ be an affine semigroup and let $X_S$ be the corresponding toric variety. We also denote by $T$ the image of the torus inside $\Aut(X_S)$. In the affine case, root subgroups are classified by non-trivial locally nilpotent derivations of $\KK[S]$ that are homogeneous with respect to the $M$-grading, see for instance \cite[Lemma~2]{L11}. Recall that a derivation $\partial\colon B\to B$ of a $\KK$-algebra $B$ is called locally nilpotent if for every $b\in B$ there exists $i\in \ZZ_{\geq 0}$ such that $\partial^i(b)=0$, where $\partial^i$ stands for the $i$-th iterate of $\partial$.

On the other hand, it is proven in \cite[Theorem 3.11]{DL24} that homogeneous locally nilpotent derivations on $\KK[S]$ are themselves classified by the Demazure roots of $S$. Indeed, letting $\alpha\in M$ be a Demazure root of $S$ the map 
$$\partial_\alpha\colon \KK[S]\to \KK[S] \mbox{ given by } \chi^m\mapsto \langle m,\rho\rangle\cdot \chi^{m+\alpha}\,,$$ where $\rho$ is the distinguished ray of $\alpha$ defines a non-trivial homogeneous locally nilpotent derivation of $\KK[S]$ and every non-trivial homogeneous locally nilpotent derivation of $\KK[S]$ equals $\lambda\partial_\alpha$ for some constant $\lambda\in \KK^*$ and some Demazure root $\alpha\in \mathcal{R}(S)$. Under this identification, the Demazure root $\alpha$ is also the degree of $\partial$ as homogeneous linear map and the weight of the root subgroup, see \cref{section-RvS23} for the definition of the weight of a root subgroup.

The following proposition states the well known fact that locally nilpotent derivations lift to the normalization.

\begin{proposition}[{\cite[Proposition 3.6]{DL24}}]\label{proposition: 3.6 DL24} Let $S$ be a non-necessarily saturated affine semigroup and let $\widetilde{S}$ be its saturation. Then every Demazure root of $S$ is also a Demazure root of $\widetilde{S}$.
\end{proposition}

\medskip

In the case where the semigroup $S$ is saturated, all the correspondences above can be translated into convex geometry in the following way \cite[Sections 1.1 and 1.2]{cox2011toric}. The bilinear pairing $\langle\ ,\ \rangle \colon M\times N\to \ZZ$ extends naturally to $\langle\ ,\ \rangle \colon M_\mathbb{Q}\times N_\mathbb{Q}\to \Q$ where $M_\mathbb{Q}:=M\otimes_{\mathbb{Z}} \Q \simeq \Q^n$ and $N_\mathbb{Q}:=N\otimes_{\mathbb{Z}} \Q \simeq \Q^n$. Now, let $\sigma^\vee$ be the cone generated by $S$ inside $M_\Q$. The saturation condition ensures that $S=\sigma^\vee\cap M$ and so we can also recover $S$ from $\sigma^\vee$. Moreover, $\sigma^\vee$ is also uniquely determined by the dual cone $\sigma=\{u\in N_\Q\mid \langle m,u\rangle\geq 0, \forall m\in \sigma^\vee\}$. This gadget allows us to obtain a correspondence between normal toric varieties and strongly convex polyhedral cones in $M_\Q$ and this correspondence can be extended into an equivalence of categories. The normal affine toric variety $X_S$ associated to a saturated semigroup $S$ is also denoted by $X_{\sigma}$ where $\sigma$ is the strongly convex rational cone previously described. Furthermore, the set of rays $S^*(1)$ of $S^*$ is also denoted by $\sigma(1)$ and in this case it coincide with the geometrical notion of ray of the cone $\sigma$. In this new setting, we keep our convention of denoting a ray $\rho\in \sigma(1)$ by its unique primitive vector.

In the case of a saturated affine semigroup $S=\sigma^\vee\cap M$ for some strongly convex polyhedral cone $\sigma\subset N_\QQ$, we can give an equivalent definition of the set of Demazure roots of $S$.

\begin{definition}[{\cite[Definition~1.5]{ArLi12}}]\label{definition-root-normal} %
Let $S$ be a saturated affine semigroup embedded minimally in a free abelian group $M$ of finite rank. An element $\alpha\in M$ is called a Demazure root of $S$ if 
\begin{enumerate}[$(i)$]
 \item There exists $\rho_\alpha\in S^*(1)$ such that $\langle \alpha, \rho_e \rangle=-1$, and
 \item $\langle \alpha, \rho \rangle\geq 0$ for all $\rho\in S^*(1)\setminus\{\rho_\alpha\}$.
\end{enumerate}
\end{definition}
The equivalence between both definitions in the case of a saturated affine semigroup follows from the fact that both of them are characterizations of the degrees of locally nilpotent derivations on $\KK[S]$ by \cite[Theorem~3.11]{DL24} and \cite[Theorem~1.6]{ArLi12}, respectively.

\subsection{Characterization of normal affine toric varieties} \label{section-RvS23}

In the proof of \cref{maintheorem}, we will apply several results from \cite{regeta2021characterizing} that were proven in the more general setting of normal spherical varieties. Let $G$ be a linear reductive algebraic group and let $B\subset G$ be a Borel subgroup.  Recall that a variety $X$ endowed with a regular $G$-action with the property that $B$ acts on $X$ with an open orbit is called $G$-spherical. Toric varieties are, in particular spherical varieties with respect to $G=B=T$. In this section, we recall the required results from \cite{regeta2021characterizing} specialized for the case of toric varieties. Recall that $T$ stands for the algebraic torus $\GM^n$.

Let $X$ be a closed subvariety of some affine space $\mathbb{A}^n$. A priori, $X$ is defined over the base field $\KK$, but in this section we assume further that  $X$ is defined over the subfield of rational numbers $\QQ$, i.e., $X$
is the zero set in $\mathbb{A}^n$ of a finite number of polynomials with coefficients in $\QQ$. Let now $\tau$ be a field
automorphism of $\KK$. Then, the $\QQ$-automorphism 
\begin{align} \label{Q-aut}
  \KK[x_1,\ldots,x_{n}]\to \KK[x_1,\ldots,x_{n}]\quad\mbox{given by}\quad \sum_{\alpha\in \ZZ_{\geq0}^n} a_\alpha x^\alpha \mapsto \sum_{\alpha\in \ZZ_{\geq0}^n} \tau(a_\alpha)x^\alpha\,,
\end{align}
where $x^\alpha$ stands for the usual multi-index notation $x^\alpha=x_1^{\alpha_1}\cdots x_n^{\alpha_n}$, induces the $\QQ$-automorphism $\mathbb{A}^n \to \mathbb{A}^n$ given by $(a_1,\dots,a_n) \mapsto (\tau(a_1),\dots,\tau(a_n))$ that sends $X$ to itself. We have then a $\QQ$-automorphism
\[
\tau_X \colon X \to X\,.
\]
It is indeed possible to define $\tau_X$ directly in terms of the algebra $\KK[X]$ of regular functions of $X$ proving that $\tau_X$ does not depend on the particular embedding of $X$ into $\mathbb{A}^n$. 

Recall that $\Aut(X)$ stands for the group of $\KK$-automorphisms of $X$. Letting $\psi\in \Aut(X)$, we have that $\tau_X \circ \psi \circ \tau_X^{-1}$ is again a $\KK$-automorphism of $X$ since $(\tau^*_X)^{-1} \circ \psi^* \circ \tau_X^*$ restricts to the identity in the base field $\KK$ by~\eqref{Q-aut}. Hence, if $\varphi \colon \Aut(X) \to
\Aut(Y)$ is a group isomorphism, then the map
\[
\Aut(X) \to \Aut(Y) , \; \psi \mapsto \varphi(\tau_X 
\circ \psi \circ \tau_X^{-1})
\]
is again a group isomorphism.

Every toric variety is, up to an isomorphism, defined over $\QQ$. In particular, the algebraic torus $T$ is defined over $\QQ$ and moreover, the action of $T$ on $X$ is also defined over $\QQ$. Hence, we have the following remark arrising from the considerations above.

\begin{remark}[{\cite[Remark 7.2]{regeta2021characterizing}}]
Let $X$ be an affine toric variety with acting torus $T$. We may and will assume that $X$, $T$ and the $T$-action $\rho\colon T\times X\to X$ are defined over $\QQ$. Let $\tau$ be a field automorphism of $\KK$. Then, the map
\[
\rho^{\tau} \colon T \times X \xrightarrow{\tau_T\times \tau_X}
T \times  X 
\xrightarrow{\rho} X \xrightarrow{\tau_X^{-1}} X
\]
is also a faithful algebraic $T$-action on $X$. Furthermore, we have the following commutative diagram of
groups, where the horizontal arrows are isomorphisms of abstract groups.
\[
	\xymatrix@=20pt{
		\Aut(X) \ar[rrr]^-{\psi \mapsto \tau_X \circ \psi \circ \tau_X^{-1}} 
		&&& \Aut(X) \\
		T \ar[rrr]^-{\tau_H} \ar[u] &&& T \ar[u]
	} 
	\]
\end{remark}

 The next statement follows from
\cite[Theorem C(1)]{regeta2021characterizing} and
\cite[Proposition 8.6]{regeta2021characterizing}.

\begin{theorem}\label{RvS23TheoremC1}Assume that 
$\mathbf{k}$ is uncountable, let $X,Y$ be irreducible affine varieties and let $\varphi\colon\Aut(X)\to\Aut(Y)$ be a group isomorphism. If $X$ is toric with acting torus $T$ that is not isomorphic to $T$, then the image $\varphi(T)$ is an algebraic subtorus of $\Aut(Y)$ of dimension $\dim T$ acting on $Y$ with a dense open orbit.
\end{theorem}

We outline here very briefly the main idea of the proof of \cref{RvS23TheoremC1}. The crucial result that is used in the proof of \cref{RvS23TheoremC1} is \cite[Theorem B and Remark 1.1]{cantat2019families} which roughly states that a closed commutative connected subgroup of $\Aut(X)$, where $X$ is an affine variety, is nested, i.e., is the union of algebraic subgroups. This result together with  the fact that the centralizer of a maximal commutative unipotent subgroup of $\Aut(X)$ coincides with its centralizer \cite[Theorem A]{regeta2021characterizing} are the main ingredients to prove the following proposition. Recall that an element $u \in \Aut(X)$ is unipotent if the closure of the subgroup of $\Aut(X)$ generated by $u$ is isomorphic to the additive group $\mathbb{G}_a$. 

\begin{proposition}[{\cite[Corollary B]{regeta2021characterizing}}]\label{RvS23CorollaryB}
Assume $\KK$ is uncountable. Let $X, Y$ be irreducible
affine varieties and let $\varphi \colon \Aut(X) \to \Aut(Y)$ be a group isomorphism. Then $\varphi$ maps unipotent elements to unipotent elements.
\end{proposition}
 
 Now, under conditions of \cref{RvS23TheoremC1} one can show that $T\subset \Aut(X)$ coincides with its centralizer (see \cite[Lemma 2.10]{kraft2021affine}) and hence one can show that $\varphi(T) \subset \Aut(Y)$ is a closed commutative subgroup. By \cite[Theorem B]{cantat2019families}, the connected component $\varphi(T)^\circ$ of $\varphi(T)$ is a union of connected algebraic subgroups of $\Aut(Y)$. But by \cref{RvS23CorollaryB}, the group $\varphi(T)^\circ$ does not contain unipotent elements which implies that $\varphi(T)^\circ$ is an algebraic subtorus of $\Aut(Y)$ of dimension less than or equal to $\dim Y$. 

Now assume $U \subset \Aut(X)$ is a root subgroup with respect to $T$. Using the relation between $T$ and $U$ in $\Aut(X)$ and the geometry of $\Aut(Y)$  the authors showed in \cite{regeta2021characterizing}
that $\varphi(U) \subset \Aut(Y)$ is a closed subgroup.
Moreover, by \cref{RvS23CorollaryB}, the group $\varphi(U)$ consists of unipotent elements and hence is connected. Applying the fact that $\varphi(T)^\circ \subset \varphi(T)$ is a subgroup of countable index and since $\varphi(T)$ acts on $\varphi(U)$ via conjugations with two orbits: the trivial orbit $\{ \id \}$ and $\varphi(U) \setminus  \{ \id \}$, one can show that the algebraic subtorus $\varphi(T)^\circ \subset \Aut(Y)$ acts on $\varphi(U)$ with countably many orbits. 
Hence, due to uncountability of $\KK$ one can show that
$\dim \varphi(U) \le \dim \varphi(T)^\circ$.  

Studying the relations between root subgroups of $\Aut(X)$  and $T \subset \Aut(X)$, the authors of \cite{regeta2021characterizing}  conclude that $\dim \varphi(U) =1$, i.e., $\varphi(U)\simeq \mathbb{G}_a$. Moreover, they also obtain that $\varphi(T) = \varphi(T)^\circ$ is an algebraic subtorus of $\Aut(Y)$ of dimension $\dim T$. This shows that $\varphi$ maps a maximal subtorus of $\Aut(X)$ to a maximal subtorus of $\Aut(Y)$ and root subgroups of $\Aut(X)$ with respect to $T$ to root subgroups of $\Aut(Y)$ with respect to $\varphi(T)$. The final step in the proof of \cref{RvS23TheoremC1} is to show that
$\varphi(T)$ acts on $Y$ with a dense orbit. This step is rather involved. We refer the reader to \cite[Section~8]{regeta2021characterizing} for the proof.

\medskip

Assume now $X$ is an affine variety endowed with a faithful and regular action of an algebraic torus $T=\spec \KK[M]$. A character of $T$ is a morphism $\chi\colon T\rightarrow \GM$ that is also a group homomorphism. Identifying $\GM$ with $\KK^*$, it is a straightforward verification that the set of all characters of $T$ forms a group isomorphic to $M$ given by $\{\chi^m\mid m\in M\}$. By abuse of notation, we denote the character group of $T$ also by $M$. In case the torus is not clear from the context, we denote it by $M_T$. 

Let now $U\simeq \GA$ be a root subgroup of $\Aut(X)$ with respect to $T$. Since $U$ is normalized by $T$, we can define a homomorphism
$$\theta\colon T\to \Aut(U)\quad\mbox{given by} \quad \gamma \longmapsto \theta_\gamma\colon U\to U \quad\mbox{with} \quad \theta_\gamma(s)= \gamma s \gamma^{-1}\,,$$
where $\Aut(U)$ is the set of algebraic group automorphisms of $U$. Since $U\simeq \GA$, we have $\Aut(U)\simeq \GM$ and so $\theta$ defines a character $\chi^m\colon T\to \GM$, for some $m\in M_T$. We define this character as the weight of the root subgroups $U$. Finally, we let $\mathcal{R}_T(X)$ be the subset of $M_T$ given by characters that appear as the weights of root subgroups of $\Aut(X)$, i.e.,
\[\mathcal{R}_T(X)=
\{ m\in M_T \mid \text{ there exists a root subgroup of  $\Aut(X)$ of weight } \chi^m \}. 
\]
In the case where $X=X_S$ is a toric variety with semigroup $S$ minimally embedded in a free abelian group $M$, we have $M_T=M$ and $\mathcal{R}_T(X)=\mathcal{R}(S)$ by \cref{section-demazure}. We are ready to state the main result from \cite{regeta2021characterizing} that we will apply in the proof of \cref{maintheorem}.

\begin{theorem}[{\cite[Theorem 9.1 (3)]{regeta2021characterizing}}]\label{RvS23Theorem9.1} Assume that $\mathbf{k}$ is uncountable, let $X,Y$ be irreducible affine varieties, let $\varphi\colon \Aut(X)\to\Aut(Y)$ be a group isomorphism.  If $X$ is a normal toric variety with acting torus $T$ that is not isomorphic to $T$, then there exists a field automorphism $\tau$ of $\mathbf{k}$ such that 
$$\nu\colon M_T\to M_{\varphi(T)}\quad\mbox{given by}\quad \chi\mapsto \chi\circ \left(\varphi^{-1}\right)|_{\varphi(T)}\circ \tau_{\varphi(T)}$$
is a well defined isomorphism and under this isomorphism, the image of $\mathcal{R}_T(X)$ is $\mathcal{R}_{\varphi(T)}(Y)$.
\end{theorem}

\begin{remark} \label{Th1.7-toric}
Under the hypothesis of \cref{RvS23Theorem9.1}, it follows from \cref{RvS23TheoremC1} that $Y$ is also toric of dimension $\dim X$ and so we can assume that $X=X_S$ and $Y=X_{S'}$ for some affine semigroups $S$ and $S'$ with $S$ saturated. Now, pick minimal embeddings  $S\hookrightarrow M$ and $S'\hookrightarrow M'$ into free abelian groups $M$ and $M'$, respectively. With this definitions, \cref{RvS23Theorem9.1} states that $\nu\colon M\to M'$ is an isomorphism that induces a bijection from $\mathcal{R}(S)$ to $\mathcal{R}(S')$. 
\end{remark}

\section{Proof of Theorem A}
\label{sectionproofoftheorem}

For the proof of \cref{maintheorem}, we need the following technical lemma that is a slight generalization of \cite[Lemma~2.4 and Remark~2.5]{L10}.

\begin{lemma} \label{famous-remark}
 Let $M$ be a free abelian group and let $S$ be an affine semigroup minimally embedded in $M$. 
 \begin{enumerate}

     \item[$(i)$] If $\alpha\in \mathcal{R}_\rho(S)$ and $m\in \rho^\bot\cap S$, then $\alpha+m\in \mathcal{R}_\rho(S)$. 
     \item[$(ii)$]  If $S$ is saturated, then $\mathcal{R}_\rho(S)\neq \emptyset$ for all $\rho\in S^*(1)$.
 \end{enumerate}
 \end{lemma}

\begin{proof}
 We first prove $(i)$. Let $\alpha\in \mathcal{R}_\rho(S)$ and let $m\in \rho^\bot\cap S$. Then, $\langle m+\alpha,\rho\rangle=\langle m,\rho\rangle+\langle \alpha,\rho\rangle=-1$. This yields \cref{definition-root}~$(i)$. Now, let $m' \in S$ be such that $\langle m',\rho\rangle >0$. Since $\alpha \in \mathcal{R}_{\rho}(S)$, we have $m'+\alpha\in S$. Hence, $m'+m+\alpha\in S$, proving that $\alpha +m\in \mathcal{R}_{\rho}(S)$. This yields \cref{definition-root}~$(ii)$, proving the first statement of the lemma.
    
 To prove $(ii)$, assume that $S$ is saturated and let $\rho \in S^*(1)$. Let $S'$ be the saturated affine semigroup whose dual $(S')^*$ has rays $S^*(1)\setminus\{\rho\}$, i.e., 
 $$S'=\big\{m\in M\mid \langle m,\rho'\rangle \geq 0 \text{ for all }\rho'\in S^*(1)\setminus\{\rho\}\big\}\,.$$ 
 By the definition of $S'$, we have $S\subset S'$. Moreover, by \cref{definition-root-normal}, we have
 $$\mathcal{R}_\rho(S)=S'\cap \{\alpha\in M\mid \langle \alpha,\rho\rangle=-1\}\,.$$
 Now, assume that $\mathcal{R}_\rho(S)=\emptyset$. By the saturation of $S'$, this implies that 
 $\langle \alpha,\rho\rangle\geq 0$ for all $\alpha\in S'$. But this contradicts the fact that $\rho$ is not a ray of $(S')^*$. We conclude that $\mathcal{R}_\rho(S)\neq\emptyset$, proving the second statement of the lemma.
\end{proof}

Let $S$ and $S'$ be pointed saturated affine semigroups minimally embedded in $M$. In \cite[Lemma 6.11]{LRU23}, we proved that if $\mathcal{R}(S)=\mathcal{R}(S')$ then $S=S'$. This shows, in particular, that a normal affine toric variety $X_S$ is completely determined by the set of its roots with respect to any fixed maximal torus in $\operatorname{Aut}(X_S)$. To prove \cref{maintheorem}, we need the following stronger statement that allows us to deal with the case of non-normal affine toric varieties.

\begin{lemma}\label{Lemma: subset distinguished ray}
Let $S$ and $S'$ be saturated affine semigroups minimally embedded in $M$ and assume that $S$ and $S'$ are both different from $M$. If $\mathcal{R}(S)\subset \mathcal{R}(S')$ then $S=S'$.
\end{lemma}

\begin{proof}

Let $\alpha\in \mathcal{R}(S)$. Then, there exists $\rho_\alpha\in S^*(1)$, the distinguished ray of $\alpha$, such that $\langle \alpha,\rho_\alpha\rangle=-1$ 
 and $\langle \alpha,\rho\rangle\geq 0$ for all $\rho\in S^*(1)\setminus\{\rho_\alpha\}$. By \cref{famous-remark}~$(i)$, the element $\alpha+m$ is also in $\mathcal{R}(S)$ for every $m\in \rho^{\bot}_\alpha\cap S$. Hence, there exists a unique hyperplane $H$ in $M_\QQ$ such that the linear span of $H\cap (\mathcal{R}(S)-\alpha)$ equals $H$. The orthogonal $H^\bot$ is a line $L$ that contains only two primitive elements $\pm p$ of $N$. The ray $\rho_\alpha$ is $-\langle \alpha,p\rangle\cdot p$ since $\langle \alpha,\rho_\alpha\rangle=-\langle \alpha,p\rangle^2=-1$.

 Furthermore, since $\mathcal{R}(S)\subset \mathcal{R}(S')$ we have that $\alpha$ is also in $\mathcal{R}(S')$. By the same analysis we obtain that there exists a unique hyperplane $H'$ in $M_\QQ$ such that the linear span of $H'\cap (\mathcal{R}(S')-\alpha)$ equals $H'$. But since 
 $$\left(\mathcal{R}(S)-\alpha\right)\subset \left(\mathcal{R}(S')-\alpha\right)\,,$$
 we conclude $H=H'$. Hence, the distinguished ray $\rho'_\alpha$ of $\alpha$ as a root of $S'$ is again given by $\rho'_\alpha=-\langle \alpha,p\rangle\cdot p$. This yields $\rho_\alpha=\rho'_\alpha$. Moreover, by \cref{famous-remark}~$(ii)$, we have that every ray $\rho$ in $S^*(1)$ and $S'^*(1)$ is the distinguished ray of some Demazure root so we have proved that 
 $$S^*(1)\subset (S')^*(1)\,.$$

We will now prove that $S^*(1)=(S')^*(1)$  which  in turn proves $S=S'$ since $S$ and $S'$ are both saturated. We proceed by contradiction. Assume that there exists $\rho'\in (S')^*(1)\setminus S^*(1)$. Pick also $\rho\in S^*(1)\subset (S')^*(1)$. Without loss of generality, we can assume that $\rho'$ and $\rho$ are in the same facet of $(S')^*$. Let $\alpha\in \mathcal{R}(S)\subset \mathcal{R}(S')$ with distinguished ray $\rho_\alpha=\rho$. Since $\rho'$ is not a ray in $S^*(1)$, there exists $m\in \rho^\bot\cap S$ such that $\langle m,\rho'\rangle<0$.  Now, by \cref{famous-remark}~$(i)$, we have that $\alpha+km\in \mathcal{R}(S)$ is also a Demazure root of $S$ with distinguished ray $\rho$ for all $k\in \ZZ_{\geq 0}$. On the other hand, we have that
 $$ \langle \alpha+km,\rho'\rangle=\langle \alpha,\rho'\rangle+k\cdot\langle m,\rho'\rangle<0\quad \mbox{for all }\quad k> -\nicefrac{\langle m,\rho'\rangle}{\langle \alpha,\rho'\rangle}.$$
 Fix any $k_0\in \ZZ_{\geq 0}$ such that $k_0 > 
 -\nicefrac{\langle m,\rho'\rangle}{\langle \alpha,\rho'\rangle}$. We conclude that 
 $$ \langle \alpha+k_0 m,\rho'\rangle<0, \quad\mbox{and}\quad \langle \alpha+k_0 m,\rho\rangle<0\,.$$
Hence, $\alpha+k_0 m\notin \mathcal{R}(S')$ and $\alpha+k_0 m\in \mathcal{R}(S)$. This is a contradiction since $\mathcal{R}(S)\subset \mathcal{R}(S')$.
\end{proof}

We now apply \cref{Lemma: subset distinguished ray} to prove that if $S$ and $S'$ have the same set of roots and $S$ is saturated, then $S$ is the normalization of $S'$. This will be the main combinatorial ingredient in the proof of \cref{maintheorem}.

\begin{lemma} \label{normalization}
 Let $M$ be a free abelian group of rank $n\in \ZZ_{>0}$ and let $S$ and $S'$ be affine semigroups embedded in $M$ minimally. Assume further that $S$ is saturated. If $\mathcal{R}(S)=\mathcal{R}(S')\neq \emptyset$, then $S$ is the saturation of $S'$. 
\end{lemma}

\begin{proof}
Since $\mathcal{R}(S)$ and $\mathcal{R}(S')$ are both different from the empty set, we have that $S$ and $S'$ are both different from $M$ since $S^*$ and $(S')^*$ have at least one ray. Now, let $\widetilde{S}'$ be the saturation of $S'$. By \cref{proposition: 3.6 DL24}, we have $\mathcal{R}(S)=\mathcal{R}(S')\subset \mathcal{R}(\widetilde{S}')$. Now, \cref{Lemma: subset distinguished ray} yields $S=\widetilde{S}'$ and so the result follows.
\end{proof}

All the saturation conditions in \cref{famous-remark}~$(i)$, \cref{Lemma: subset distinguished ray} and \cref{normalization} are essential as the following example shows.

\begin{example} \label{norm-required}
Letting $M=\ZZ^2$, we let $\sigma_1,\sigma_2\subset M_\QQ$ be the cones $\sigma_1=\cone((0,1),(2,-1))$ and $\sigma_2=\cone((0,1),(3,-1))$, respectively. Their dual cones are $\sigma_1^\vee=\cone((1,0),(1,2))$ and $\sigma_2^\vee=\cone((1,0),(1,3))$, respectively. The corresponding affine semigroups $S_1$ and $S_2$ are generated by the sets $\{(1,0),(1,1),(1,2)\}$ and $\{(1,0),(1,1),(1,2), (1,3)\}$, respectively. Let $S'_1=S_1\setminus \{(2k-1,4k-2)\mid k\in \ZZ_{>0}\}$ and $S'_2=S_2\setminus \{(2k-1,6k-3)\mid k\in \ZZ_{>0}\}$ presented in Figure 1 and Figure 2, respectively. We have
$$\mathcal{R}(S'_1)=\mathcal{R}(S'_2)=\{(k,-1)\mid k\in \ZZ_{\ge0}\}\,.$$
Nevertheless, the saturations of $S_1'$ and $S_2'$ are $S_1$ and $S_2$, respectively. Hence the saturations of $S_1'$ and $S_2'$ do not coincide.
$$
\begin{array}{cc}
\begin{picture}(130,90)
\definecolor{gray1}{gray}{0.7}
\definecolor{gray2}{gray}{0.85}
\definecolor{green}{RGB}{0,124,0}

\textcolor{gray2}{\put(12,15.2){\vector(1,0){80}}}
\textcolor{gray2}{\put(20,5){\vector(0,1){80}}}

\put(10,5){\textcolor{gray1}{\circle*{3}}}
\put(20,5){\textcolor{green}{\circle*{3}}}
\put(30,5){\textcolor{green}{\circle*{3}}} 
\put(40,5){\textcolor{green}{\circle*{3}}} 
\put(50,5){\textcolor{green}{\circle*{3}}} 
\put(60,5){\textcolor{green}{\circle*{3}}} 
\put(70,5){\textcolor{green}{\circle*{3}}} 
\put(80,5){\textcolor{green}{\circle*{3}}} 
\put(90,5){\textcolor{green}{\circle*{3}}}

\put(10,15){\textcolor{gray1}{\circle*{3}}}
\put(20,15){\circle*{3}} 
\put(30,15){\circle*{3}} 
\put(40,15){\circle*{3}} 
\put(50,15){\circle*{3}}
\put(60,15){\circle*{3}}
\put(70,15){\circle*{3}}
\put(80,15){\circle*{3}}
\put(90,15){\circle*{3}}

\put(10,25){\textcolor{gray1}{\circle*{3}}}
\put(20,25){\textcolor{gray1}{\circle*{3}}}
\put(30,25){\circle*{3}} 
\put(40,25){\circle*{3}} 
\put(50,25){\circle*{3}}
\put(60,25){\circle*{3}}
\put(70,25){{\circle*{3}}} 
\put(80,25){\circle*{3}}
\put(90,25){\circle*{3}}

\put(10,35){\textcolor{gray1}{\circle*{3}}}
\put(20,35){\textcolor{gray1}{\circle*{3}}}
\put(30,35){{\circle{3}}}
\put(40,35){\circle*{3}} 
\put(50,35){{\circle*{3}}}
\put(60,35){{\circle*{3}}}
\put(70,35){\circle*{3}} 
\put(80,35){\circle*{3}}
\put(90,35){\circle*{3}}

\put(10,45){\textcolor{gray1}{\circle*{3}}}
\put(20,45){\textcolor{gray1}{\circle*{3}}}
\put(30,45){\textcolor{gray1}{\circle*{3}}}
\put(40,45){{\circle*{3}}}
\put(50,45){\circle*{3}}
\put(60,45){\circle*{3}}
\put(70,45){\circle*{3}} 
\put(80,45){\circle*{3}} 
\put(90,45){\circle*{3}}

\put(10,55){\textcolor{gray1}{\circle*{3}}}
\put(20,55){\textcolor{gray1}{\circle*{3}}}
\put(30,55){\textcolor{gray1}{\circle*{3}}}
\put(40,55){{\circle*{3}}}
\put(50,55){\circle*{3}}
\put(60,55){\circle*{3}}
\put(70,55){\circle*{3}} 
\put(80,55){\circle*{3}}
\put(90,55){\circle*{3}}

\put(10,65){\textcolor{gray1}{\circle*{3}}}
\put(20,65){\textcolor{gray1}{\circle*{3}}}
\put(30,65){\textcolor{gray1}{\circle*{3}}}
\put(40,65){\textcolor{gray1}{\circle*{3}}}
\put(50,65){{\circle*{3}}}
\put(60,65){\circle*{3}}
\put(70,65){\circle*{3}} 
\put(80,65){\circle*{3}} 
\put(90,65){\circle*{3}}

\put(10,75){\textcolor{gray1}{\circle*{3}}}
\put(20,75){\textcolor{gray1}{\circle*{3}}}
\put(30,75){\textcolor{gray1}{\circle*{3}}}
\put(40,75){\textcolor{gray1}{\circle*{3}}}
\put(50,75){{\circle{3}}}
\put(60,75){{\circle*{3}}}
\put(70,75){\circle*{3}} 
\put(80,75){\circle*{3}} 
\put(90,75){\circle*{3}} 

\put(10,85){\textcolor{gray1}{\circle*{3}}}
\put(20,85){\textcolor{gray1}{\circle*{3}}}
\put(30,85){\textcolor{gray1}{\circle*{3}}}
\put(40,85){\textcolor{gray1}{\circle*{3}}}
\put(50,85){\textcolor{gray1}{\circle*{3}}}
\put(60,85){{\circle*{3}}}
\put(70,85){{\circle*{3}}}
\put(80,85){\circle*{3}} 
\put(90,85){\circle*{3}} 

\end{picture}&\begin{picture}(130,90)
\definecolor{gray1}{gray}{0.7}
\definecolor{gray2}{gray}{0.85}
\definecolor{green}{RGB}{0,124,0}

\textcolor{gray2}{\put(12,15.2){\vector(1,0){80}}}
\textcolor{gray2}{\put(20,5){\vector(0,1){80}}}

\put(10,5){\textcolor{gray1}{\circle*{3}}}
\put(20,5){\textcolor{green}{\circle*{3}}}
\put(30,5){\textcolor{green}{\circle*{3}}} 
\put(40,5){\textcolor{green}{\circle*{3}}} 
\put(50,5){\textcolor{green}{\circle*{3}}} 
\put(60,5){\textcolor{green}{\circle*{3}}} 
\put(70,5){\textcolor{green}{\circle*{3}}} 
\put(80,5){\textcolor{green}{\circle*{3}}} 
\put(90,5){\textcolor{green}{\circle*{3}}}

\put(10,15){\textcolor{gray1}{\circle*{3}}}
\put(20,15){\circle*{3}} 
\put(30,15){\circle*{3}} 
\put(40,15){\circle*{3}} 
\put(50,15){\circle*{3}}
\put(60,15){\circle*{3}}
\put(70,15){\circle*{3}}
\put(80,15){\circle*{3}}
\put(90,15){\circle*{3}}

\put(10,25){\textcolor{gray1}{\circle*{3}}}
\put(20,25){\textcolor{gray1}{\circle*{3}}}
\put(30,25){\circle*{3}} 
\put(40,25){\circle*{3}} 
\put(50,25){\circle*{3}}
\put(60,25){\circle*{3}}
\put(70,25){{\circle*{3}}} 
\put(80,25){\circle*{3}}
\put(90,25){\circle*{3}}

\put(10,35){\textcolor{gray1}{\circle*{3}}}
\put(20,35){\textcolor{gray1}{\circle*{3}}}
\put(30,35){{\circle*{3}}}
\put(40,35){\circle*{3}} 
\put(50,35){{\circle*{3}}}
\put(60,35){{\circle*{3}}}
\put(70,35){\circle*{3}} 
\put(80,35){\circle*{3}}
\put(90,35){\circle*{3}}

\put(10,45){\textcolor{gray1}{\circle*{3}}}
\put(20,45){\textcolor{gray1}{\circle*{3}}}
\put(30,45){{\circle{3}}}
\put(40,45){{\circle*{3}}}
\put(50,45){\circle*{3}}
\put(60,45){\circle*{3}}
\put(70,45){\circle*{3}} 
\put(80,45){\circle*{3}} 
\put(90,45){\circle*{3}} 

\put(10,55){\textcolor{gray1}{\circle*{3}}}
\put(20,55){\textcolor{gray1}{\circle*{3}}}
\put(30,55){\textcolor{gray1}{\circle*{3}}}
\put(40,55){{\circle*{3}}}
\put(50,55){\circle*{3}}
\put(60,55){\circle*{3}}
\put(70,55){\circle*{3}} 
\put(80,55){\circle*{3}}
\put(90,55){\circle*{3}}

\put(10,65){\textcolor{gray1}{\circle*{3}}}
\put(20,65){\textcolor{gray1}{\circle*{3}}}
\put(30,65){\textcolor{gray1}{\circle*{3}}}
\put(40,65){{\circle*{3}}}
\put(50,65){{\circle*{3}}}
\put(60,65){\circle*{3}}
\put(70,65){\circle*{3}} 
\put(80,65){\circle*{3}} 
\put(90,65){\circle*{3}}

\put(10,75){\textcolor{gray1}{\circle*{3}}}
\put(20,75){\textcolor{gray1}{\circle*{3}}}
\put(30,75){\textcolor{gray1}{\circle*{3}}}
\put(40,75){{\circle*{3}}}
\put(50,75){{\circle*{3}}}
\put(60,75){{\circle*{3}}}
\put(70,75){\circle*{3}} 
\put(80,75){\circle*{3}} 
\put(90,75){\circle*{3}}

\put(10,85){\textcolor{gray1}{\circle*{3}}}
\put(20,85){\textcolor{gray1}{\circle*{3}}}
\put(30,85){\textcolor{gray1}{\circle*{3}}}
\put(40,85){\textcolor{gray1}{\circle*{3}}}
\put(50,85){{\circle*{3}}}
\put(60,85){{\circle*{3}}}
\put(70,85){{\circle*{3}}}
\put(80,85){\circle*{3}} 
\put(90,85){\circle*{3}} 

\end{picture}\\
\text{Figure 1}&\text{Figure 2}\\
\end{array}
$$

\end{example}

In the proof of \cref{maintheorem}, we will apply  the following proposition.

\begin{proposition}\label{A1timesZisdeterminedbyroots}
Assume $X_\sigma$ is a normal affine toric variety that is isomorphic to $\mathbb{A}^k \times X_{\sigma'}$, where $X_{\sigma'}$ is a normal affine toric variety and $k\in \ZZ_{>0}$. If $X_S$ is a non-necessarily normal toric variety such that $\mathcal{R}(S)=\mathcal{R}(\sigma^\vee\cap M)$, then $X_S$ is normal and, moreover, $X_S \simeq X_\sigma$. 
\end{proposition}

\begin{proof} 
If $M$ is the free abelian group of characters associated to $X_\sigma\simeq \AF^k\times X_{\sigma'}$, the free abelian group of characters associated to $X_{\sigma'}$ will be denoted by $M'$ so that $M=\ZZ^k\oplus M'$. By \cref{normalization}, the toric variety $X_\sigma$ is the normalization of $X_S$. Note first that since $(e^*_i,0)\in \sigma^\vee\cap M$, the element $(le^*_i,0)\in S$ for some $l\in \ZZ_{>0}$. Furthermore, by hypothesis, the element $(-e^*_i,0)$ is a Demazure root of $S$, so we conclude that $(e^*_i,0)\in S$. Let now $m'\in (\sigma')^\vee\cap M'$. A straightforward verification shows that $(-e^*_i,m')$ is also a Demazure root of $S$. Therefore, $(e^*_i,0)+(-e^*_i,m')=(0,m')\in S$. Since $(0,m')\in S$ for all $m'\in(\sigma')^\vee\cap M'$, and $(e^*_i,0)\in S$ for all $i$, we have $S=\sigma^\vee\cap M$.
\end{proof}

For the proof of \cref{isomofaut}, we need the following lemma. Recall that an affine toric variety $X$ is called non-degenerate if  it is not  isomorphic to $\GM\times Y$, where $Y$ is an affine toric variety.

\begin{lemma} \label{fix-point-is-fixed}
Assume $X_\sigma$ is a non-degenerate normal affine toric variety with acting torus $T$ that is not isomorphic to $\mathbb{A}^1 \times Y$ for any affine toric variety $Y$ and let $x_0 \in X_\sigma$ be the unique fixed point with respect to the $T$-action. If $\varphi\colon X_\sigma\to X_\sigma$ is an automorphism of $X_\sigma$, then $\varphi(x_0)=x_0$.
\end{lemma}

\begin{proof}
 Recall that by \cite[Chapter~2, Section~1, Corollary~1~$iii)$]{KKMS73}, the local rings at the fixed points of two non-degenerate normal affine toric varieties $X_\sigma$ and $X_{\sigma'}$ are isomorphic if and only if $\sigma\simeq\sigma'$ where the isomorphism is of rational polyhedral cones. Let now $x_1\in X_\sigma$ be a point different from $x_0$. Then $x_1$ belongs to a positive-dimensional orbit and so the local ring of $x_1$ is isomorphic to the local ring at the fixed point of a toric variety $X_{\sigma'}$ that is isomorphic to a product $\mathbb{A}^1 \times Y$. Since $X$ is not isomorphic to $\mathbb{A}^1 \times Y$, it follows that $\sigma$ is not isomorphic to $\sigma'$. Hence, the local rings of $x_0$ and $x_1$ are not isomorphic. This yields $\varphi(x_0)\neq x_1$ and so $\varphi(x_0)=x_0$.
\end{proof}

Let now $S$ be an affine semigroup. An element $m\in S$ is called irreducible if $m=m'+m''$ with $m',m'' \in S$ implies that either $m'$ or $m''$ equals $0$. If $S$ is pointed, the set of irreducible elements of $S$ is denoted by $\mathcal{H}$ and is called the Hilbert basis of $S$. 

\begin{definition} 
Let $S$ be a pointed affine semigroup and let $\mathcal{H}$ be the Hilbert basis of $S$. An element $m\in S$ is said to be $l$-decomposable if there exist $h_1,\ldots, h_l\in \mathcal{H}$ such that $m=h_1+\ldots+h_l$ and if $m=h'_1+\ldots +h'_{l'}$ with $h'_1,\ldots,h'_{l'}\in \mathcal{H}$ then $l'\le l$. The set of $l$-decomposable elements of $S$ is denoted by $\mathcal{H}_l$. We also let $\widehat{\mathcal{H}}_l=\bigcup_{i=1}^l\mathcal{H}_i$.
\end{definition}

For $l=1$ we have $\mathcal{H}_1=\widehat{\mathcal{H}}_1=\mathcal{H}$. Moreover, we have the following lemma.

\begin{lemma}
 Let $S$ be a pointed semigroup. The set $S_l:=S\setminus \widehat{\mathcal{H}}_l$ is an affine subsemigroup of $S$. Moreover, $S$ and $S_l$ have the same saturation $\widetilde{S}$. 
\end{lemma}

\begin{proof} 
Let $m=h_1+\ldots+h_r$ and $m'=h'_1+\ldots +h'_{r'}$ in $S_l$ with $h_1,\ldots,h_r,h'_1,\ldots,h'_{r'}$ in $\H$. By definition of $S_l$ we have $r,r'>l$. Now the sum $m+m'=h_1+\ldots+h_r+h'_1+\ldots +h'_{r'}$ is also contained in $S_l$ since $r+r'>l$. Furthermore, $S_l$ is finitely generated and the saturations of $S$ and $S_l$ coincide since $\widehat{\mathcal{H}}_l$ is finite.
\end{proof}

We will now prove the following proposition.

\begin{proposition}\label{isomofaut}
Let $X_\sigma$ be a normal non-degenerate affine toric variety not isomorphic to the algebraic torus. If further $X_\sigma$ is not isomorphic to a product $\mathbb{A}^1 \times Y$ for a toric variety $Y$, then for any positive integer $l$ we have 
$$\Aut(\GM^k\times X_\sigma) \simeq \Aut(\GM^k\times X_{S_l})\mbox{ as an ind-group}\,,$$
where $k\in \ZZ_{>0}$, $S=\sigma^{\vee}\cap M$ and $S_l=S\setminus\widehat{\mathcal{H}}_l$.
\end{proposition}

\begin{proof}
Since the normalization of $\GM^k \times X_{S_l} $ is $\GM^k\times X_\sigma$, each automorphism of $\GM^k\times X_{S_l}$ lifts to an automorphism of $\GM^k\times X_\sigma$. Moreover, by \cite[Proposition 12.1.1]{furter2018geometry}
the natural embedding $\psi \colon \Aut(\GM^k\times X_{S_l}) \to \Aut(\GM^k\times X_\sigma)$ is  a closed immersion of ind-groups. 
To prove the proposition, it suffices to show that $\psi$ is surjective or equivalently that any automorphism of $\GM^k\times X_\sigma$ induces an automorphism of $\GM^k\times X_{S_l}$.

The toric variety $\GM^k\times X_\sigma $ is given by the cone $(0,\sigma)$ in the free abelian group of characters
$M' = \mathbb{Z}^k\oplus M$. Let $T=\spec\KK[M']$.
The $T$-action on $\GM^k\times X_\sigma$ induces an action on $X_\sigma$ and $T$ acts on $X_\sigma$ with an open orbit. The variety $X_\sigma$ contains a unique $T$-fixed point. Moreover, the fixed point $x_0$ is singular since $X$ itself is singular by the hypothesis of not being isomorphic to $\mathbb{A}^1 \times Y$ nor to the algebraic torus. Furthermore, by \cref{fix-point-is-fixed},
an automorphism $\varphi \in \Aut(\GM^k\times X_\sigma)$ fixes $\GM^k\times \{ x_0 \}$. Note that the ideal
\begin{align*} 
 I(\GM^k\times \{x_0\} ) \subset \KK[\GM^k] \otimes_\KK \KK[\sigma^\vee \cap M]
 = \left(\bigoplus_{m\in \ZZ^k} \KK \chi^m \right)\otimes_\KK\left(\bigoplus_{m'\in S} \KK \chi^{m'} \right)
\end{align*}
is generated by $\left\{ \chi^m\otimes_\KK \chi^{m'} \mid m\in \mathbb{Z}^k \mbox{ and } m' \in \mathcal{H}\right\}$, where $\mathcal{H}$ is the Hilbert basis of $S$. 
Therefore, $\varphi^*(\chi^m\otimes_\KK \chi^{m'}) \in I(\GM^k \times \{ x_0 \})$ for all $m \in \mathbb{Z}^k$ and all $m' \in \mathcal{H}$.

We claim that an automorphism $\varphi^* \in \Aut(\KK[\GM^k] \otimes_\KK \KK[S])$ preserves $$\KK[\GM^k] \otimes_\KK \KK[S_l] = \left(\bigoplus_{m\in \ZZ^k} \KK \chi^m \right) \otimes_\KK \left(\bigoplus_{m' \in S_l} \KK \chi^{m'}\right)\qquad \mbox{for all } l\in \ZZ_{\geq 0}\,.$$ 
To prove this, it is enough to show that $\varphi^*(\chi^{m} \otimes \chi^{m'}) \in \KK[\GM^k] \otimes_\KK \KK[S_l]$ for any $m \in \mathbb{Z}^k, m' \in S_l$. If $m' \in S_l$, then $m'$ can be written as $m'_1 + \ldots + m'_r$, where $m'_i \in \mathcal{H}$ and $r \ge l$. Hence, 
\[\chi^m \otimes \chi^{m'} = (\chi^m \otimes \chi^{m'_1})(1 \otimes \chi^{m'_2}) \cdots (1 \otimes \chi^{m'_r})
\]
and 
\[\varphi^*(\chi^m \otimes \chi^{m'}) = \varphi^*(\chi^m \otimes \chi^{m'_1}) \varphi^*(1 \otimes \chi^{m'_2}) \cdots \varphi^*(1 \otimes \chi^{m'_r}).
\]
Since $\varphi^*$ preserves $I(\GM^k \times \{ x_0 \})$ by \cref{fix-point-is-fixed}, we have 
\begin{align*}
\varphi^*(\chi^{m} \otimes \chi^{m'_i}) = \sum_{j} \alpha_{ij}\chi^{m_j} \otimes \chi^{m'_{ij}} \in \left(\bigoplus_{m\in \ZZ^k} \KK \chi^m \right) \otimes_\KK \left(\bigoplus_{m' \in S\setminus \{0\}} \KK \chi^{m'}\right)\,. 
\end{align*}
Furthermore, since every element in $S\setminus\{0\}$ is a sum of a positive number of elements in $\mathcal{H}$, we have that each summand in the product $\varphi^*(\chi^m \otimes \chi^{m'_1}) \varphi^*(1 \otimes \chi^{m'_2}) \cdots \varphi^*(1 \otimes \chi^{m'_r})$ 
is a product of at least $r$ monomials $\chi^a\otimes\chi^b$ with $a\in \ZZ^k$ and $b\in \mathcal{H}$. Since $r \ge l$, we conclude that
\[
\varphi^*(\chi^m \otimes \chi^{m'_1}) \varphi^*(1 \otimes \chi^{m'_2}) \cdots \varphi^*(1 \otimes \chi^{m'_r}) \in \KK[\GM^k]\otimes_\KK \KK[S_l].
\]
Therefore, $\varphi^* \in \Aut(\KK[\GM^k] \otimes_\KK \KK[S])$
 induces an automorphism of $\KK[\GM^k] \otimes_\KK \KK[S_l]$ which proves the statement.
\end{proof}

\cref{isomofaut} shows, in particular, that $\Aut(X_\sigma) \simeq \Aut(X_{S_l})$ whenever $X_\sigma$ is a non-degenerate normal affine toric variety that is not isomorphic to $\AF^1\times Y$, $S=\sigma^\vee\cap M$ and $S_l=S\setminus\widehat{\mathcal{H}}_l$, for all $l\in \ZZ_{\geq 0}$. However, the next example shows that in general it is not true that if $X_{S}$ is a non-normal non-degenerate affine toric variety with normalization $X_\sigma$ such that $\Aut(X_{S})\simeq\Aut(X_{\sigma})$, then $S=S_l$, for some $l$.

\begin{example} 
Let $\widetilde{S}$ be the saturated affine semigroup generated by the set $\{(1,0),(1,1),(3,4)\}$ in $\mathbb{Z}^2$, see Figure 3. Let also $S=\widetilde{S}\setminus\{(1,0),(1,1),(3,2),(3,3),(3,4),(5,6)\}$, see Figure 4. The following images show that $S$ and $\widetilde{S}$ have the same Demazure roots and \cite[Lemma~4.9]{DL24} shows that having the same set of Demazure roots implies that $\Aut(X_S)=\Aut(X_{\widetilde{S}})$. In the figures, the black dots represent the elements of the semigroups $\widetilde{S}$ and $S$ respectively and the green dots represent the Demazure roots of both affine semigroups.
$$
\begin{array}{cc}
\begin{picture}(130,90)
\definecolor{gray1}{gray}{0.7}
\definecolor{gray2}{gray}{0.85}
\definecolor{green}{RGB}{0,124,0}

\textcolor{gray2}{\put(12,15.2){\vector(1,0){80}}}
\textcolor{gray2}{\put(20,5){\vector(0,1){80}}}

\put(60,35){{\circle*{3}}}
\put(70,35){\circle*{3}} 
\put(80,35){\circle*{3}}
\put(90,35){\circle*{3}}

\put(10,45){\textcolor{gray1}{\circle*{3}}}
\put(20,45){\textcolor{gray1}{\circle*{3}}}
\put(30,45){\textcolor{gray1}{\circle*{3}}}
\put(40,45){\textcolor{green}{\circle*{3}}}
\put(50,45){\circle*{3}}
\put(60,45){\circle*{3}}
\put(70,45){\circle*{3}} 
\put(80,45){\circle*{3}} 
\put(90,45){\circle*{3}}

\put(10,55){\textcolor{gray1}{\circle*{3}}}
\put(20,55){\textcolor{gray1}{\circle*{3}}}
\put(30,55){\textcolor{gray1}{\circle*{3}}}
\put(40,55){\textcolor{gray1}{\circle*{3}}}
\put(50,55){\circle*{3}}
\put(60,55){\circle*{3}}
\put(70,55){\circle*{3}} 
\put(80,55){\circle*{3}}
\put(90,55){\circle*{3}}
\put(10,5){\textcolor{gray1}{\circle*{3}}}
\put(20,5){\textcolor{green}{\circle*{3}}}
\put(30,5){\textcolor{green}{\circle*{3}}} 
\put(40,5){\textcolor{green}{\circle*{3}}} 
\put(50,5){\textcolor{green}{\circle*{3}}} 
\put(60,5){\textcolor{green}{\circle*{3}}} 
\put(70,5){\textcolor{green}{\circle*{3}}} 
\put(80,5){\textcolor{green}{\circle*{3}}} 
\put(90,5){\textcolor{green}{\circle*{3}}}

\put(10,15){\textcolor{gray1}{\circle*{3}}}
\put(20,15){\circle*{3}} 
\put(30,15){\circle*{3}} 
\put(40,15){\circle*{3}} 
\put(50,15){\circle*{3}}
\put(60,15){\circle*{3}}
\put(70,15){\circle*{3}}
\put(80,15){\circle*{3}}
\put(90,15){\circle*{3}}

\put(10,25){\textcolor{gray1}{\circle*{3}}}
\put(20,25){\textcolor{gray1}{\circle*{3}}}
\put(30,25){\circle*{3}} 
\put(40,25){\circle*{3}} 
\put(50,25){\circle*{3}}
\put(60,25){\circle*{3}}
\put(70,25){{\circle*{3}}} 
\put(80,25){\circle*{3}}
\put(90,25){\circle*{3}}

\put(10,35){\textcolor{gray1}{\circle*{3}}}
\put(20,35){\textcolor{gray1}{\circle*{3}}}
\put(30,35){\textcolor{gray1}{\circle*{3}}}
\put(40,35){\circle*{3}} 
\put(50,35){{\circle*{3}}}

\put(10,65){\textcolor{gray1}{\circle*{3}}}
\put(20,65){\textcolor{gray1}{\circle*{3}}}
\put(30,65){\textcolor{gray1}{\circle*{3}}}
\put(40,65){\textcolor{gray1}{\circle*{3}}}
\put(50,65){\textcolor{gray1}{\circle*{3}}}
\put(60,65){\circle*{3}}
\put(70,65){\circle*{3}} 
\put(80,65){\circle*{3}} 
\put(90,65){\circle*{3}}

\put(10,75){\textcolor{gray1}{\circle*{3}}}
\put(20,75){\textcolor{gray1}{\circle*{3}}}
\put(30,75){\textcolor{gray1}{\circle*{3}}}
\put(40,75){\textcolor{gray1}{\circle*{3}}}
\put(50,75){\textcolor{gray1}{\circle*{3}}}
\put(60,75){\textcolor{gray1}{\circle*{3}}}
\put(70,75){\circle*{3}} 
\put(80,75){\circle*{3}} 
\put(90,75){\circle*{3}} 

\put(10,85){\textcolor{gray1}{\circle*{3}}}
\put(20,85){\textcolor{gray1}{\circle*{3}}}
\put(30,85){\textcolor{gray1}{\circle*{3}}}
\put(40,85){\textcolor{gray1}{\circle*{3}}}
\put(50,85){\textcolor{gray1}{\circle*{3}}}
\put(60,85){\textcolor{gray1}{\circle*{3}}}
\put(70,85){\textcolor{green}{\circle*{3}}}
\put(80,85){\circle*{3}} 
\put(90,85){\circle*{3}} 

\end{picture}&\begin{picture}(130,90)
\definecolor{gray1}{gray}{0.7}
\definecolor{gray2}{gray}{0.85}
\definecolor{green}{RGB}{0,124,0}

\textcolor{gray2}{\put(12,15.2){\vector(1,0){80}}}
\textcolor{gray2}{\put(20,5){\vector(0,1){80}}}

\put(10,5){\textcolor{gray1}{\circle*{3}}}
\put(20,5){\textcolor{green}{\circle*{3}}}
\put(30,5){\textcolor{green}{\circle*{3}}} 
\put(40,5){\textcolor{green}{\circle*{3}}} 
\put(50,5){\textcolor{green}{\circle*{3}}} 
\put(60,5){\textcolor{green}{\circle*{3}}} 
\put(70,5){\textcolor{green}{\circle*{3}}} 
\put(80,5){\textcolor{green}{\circle*{3}}} 
\put(90,5){\textcolor{green}{\circle*{3}}}

\put(10,15){\textcolor{gray1}{\circle*{3}}}
\put(20,15){\circle*{3}} 
\put(30,15){\circle{3}} 
\put(40,15){\circle*{3}} 
\put(50,15){\circle*{3}}
\put(60,15){\circle*{3}}
\put(70,15){\circle*{3}}
\put(80,15){\circle*{3}}
\put(90,15){\circle*{3}}

\put(10,25){\textcolor{gray1}{\circle*{3}}}
\put(20,25){\textcolor{gray1}{\circle*{3}}}
\put(30,25){\circle{3}} 
\put(40,25){\circle*{3}} 
\put(50,25){\circle*{3}}
\put(60,25){\circle*{3}}
\put(70,25){{\circle*{3}}} 
\put(80,25){\circle*{3}}
\put(90,25){\circle*{3}}

\put(10,35){\textcolor{gray1}{\circle*{3}}}
\put(20,35){\textcolor{gray1}{\circle*{3}}}
\put(30,35){\textcolor{gray1}{\circle*{3}}}
\put(40,35){\circle*{3}} 
\put(50,35){{\circle{3}}}
\put(60,35){{\circle*{3}}}
\put(70,35){\circle*{3}} 
\put(80,35){\circle*{3}}
\put(90,35){\circle*{3}}

\put(10,45){\textcolor{gray1}{\circle*{3}}}
\put(20,45){\textcolor{gray1}{\circle*{3}}}
\put(30,45){\textcolor{gray1}{\circle*{3}}}
\put(40,45){\textcolor{green}{\circle*{3}}}
\put(50,45){\circle{3}}
\put(60,45){\circle*{3}}
\put(70,45){\circle*{3}} 
\put(80,45){\circle*{3}} 
\put(90,45){\circle*{3}}

\put(10,55){\textcolor{gray1}{\circle*{3}}}
\put(20,55){\textcolor{gray1}{\circle*{3}}}
\put(30,55){\textcolor{gray1}{\circle*{3}}}
\put(40,55){\textcolor{gray1}{\circle*{3}}}
\put(50,55){\circle{3}}
\put(60,55){\circle*{3}}
\put(70,55){\circle*{3}} 
\put(80,55){\circle*{3}}
\put(90,55){\circle*{3}}

\put(10,65){\textcolor{gray1}{\circle*{3}}}
\put(20,65){\textcolor{gray1}{\circle*{3}}}
\put(30,65){\textcolor{gray1}{\circle*{3}}}
\put(40,65){\textcolor{gray1}{\circle*{3}}}
\put(50,65){\textcolor{gray1}{\circle*{3}}}
\put(60,65){\circle*{3}}
\put(70,65){\circle*{3}} 
\put(80,65){\circle*{3}} 
\put(90,65){\circle*{3}}

\put(10,75){\textcolor{gray1}{\circle*{3}}}
\put(20,75){\textcolor{gray1}{\circle*{3}}}
\put(30,75){\textcolor{gray1}{\circle*{3}}}
\put(40,75){\textcolor{gray1}{\circle*{3}}}
\put(50,75){\textcolor{gray1}{\circle*{3}}}
\put(60,75){\textcolor{gray1}{\circle*{3}}}
\put(70,75){\circle{3}} 
\put(80,75){\circle*{3}} 
\put(90,75){\circle*{3}} 

\put(10,85){\textcolor{gray1}{\circle*{3}}}
\put(20,85){\textcolor{gray1}{\circle*{3}}}
\put(30,85){\textcolor{gray1}{\circle*{3}}}
\put(40,85){\textcolor{gray1}{\circle*{3}}}
\put(50,85){\textcolor{gray1}{\circle*{3}}}
\put(60,85){\textcolor{gray1}{\circle*{3}}}
\put(70,85){\textcolor{green}{\circle*{3}}}
\put(80,85){\circle*{3}} 
\put(90,85){\circle*{3}} 

\end{picture}\\
\text{Figure 3}&\text{Figure 4}\\
\end{array}
$$
\end{example}

\begin{remark} \label{all-S}
 Given a saturated affine semigroup $S$, it seems that finding all affine semigroups $S'$ that have $S$ as its saturation and such that $\mathcal{R}(S')=\mathcal{R}(S)$ is a hard combinatorial problem and is currently beyond our reach.
\end{remark}

The case of the algebraic torus is handled by \cref{main-torus}. We first need the following remark, borrowed from \cite{regeta2021characterizing}, proving that any variety with automorphism group isomorphic to the automorphism group of the algebraic torus $T$, admits an action of $T$.

\begin{remark}\label{remarkalgebraictorus}
 Let $T$ be the algebraic torus of dimension $n$ and let $Y$ be an affine irreducible variety of dimension $n$. Fix an isomorphism \[
 \varphi\colon \Aut(T) \to \Aut(Y).
 \]
By \cite[Lemma 2.10]{kraft2021affine}, the torus $T\subset \Aut(T)$ coincides with its centralizer in $\Aut(T)$. Hence, $\varphi(T)\subset \Aut(Y)$ coincides with its centralizer in $\Aut(Y)$ which implies that $\varphi(T) \subset \Aut(Y)$ is a closed subgroup. 
 Moreover, applying the argument as in the proof of \cite[Theorem E]{regeta2021characterizing}, we conclude that $\varphi(T)^\circ \subset \Aut(Y)$ is the algebraic torus of dimension $n$.
\end{remark}

We can now handle the case of the algebraic torus.

\begin{proposition} \label{main-torus}
Let $T$ be the algebraic torus and let $Y$ be an irreducible affine variety with $ \dim Y\leq \dim T$. If $\Aut(Y) \simeq \Aut(T)$ as a group, then $Y \simeq T$ as a variety.
\end{proposition}
\begin{proof}
Let $T$ be the algebraic torus of dimension $n$. Without loss of generality we may assume that $n\ge 1$.
Fix an isomorphism $\varphi\colon \Aut(T) \to \Aut(Y)$. 
It is known that 
\begin{align} \label{eq:aut-torus}
 \Aut(T) = T\rtimes \GL(M)\,,
\end{align}
where $M\simeq \ZZ^n$ is the free abelian group of characters of $T$ and so $\GL(M)\simeq \GL_n(\mathbb{Z})$ is the group of lattice automorphisms of $M$. By \cref{remarkalgebraictorus}, the image of $T$ in $\Aut(Y)$ is a closed subgroup and $\varphi(T)^\circ \subset \Aut(Y)$ is the algebraic torus of dimension $n$, i.e., $Y$ is an affine toric variety. Moreover, since $\varphi(T)^\circ \subset \Aut(Y)$ is a maximal subtorus of dimension $n=\dim Y$, we have that $\varphi(T)^\circ \subset \Aut(Y)$ coincides with its centralizer by \cite[Lemma 2.10]{kraft2021affine}. Hence, we obtain $\varphi^{-1}(\varphi(T)^\circ) = T$ which implies that $\varphi(T) = \varphi(T)^\circ$ and so $\varphi(T)$ is the acting torus of $Y$.

By \eqref{eq:aut-torus}, we have that $\Aut(Y)$ is a countable extension of the algebraic torus $\varphi(T)$. Hence, the group $\Aut(Y)$ does not contain any unipotent algebraic subgroup since $\GA=(\KK,+)$ is uncountable. Such $Y$ is called rigid. We conclude that $Y$ is a not necessarily normal rigid affine toric variety with an action of $n$-dimensional torus $\varphi(T)$. Denote by $M_Y$ the free abelian group of characters of $\varphi(T)$ and
let now $S\subset M_Y$ be the affine semigroup of $Y$. 
 
By \cite[Theorem~3]{BoGa21} we have that 
\[
\Aut(Y)=\varphi(T) \rtimes G,
\]
where $G$ is the subgroup of $\GL(M_Y) \simeq \GL_n(\mathbb{Z})$ fixing $S$. Remark that $G$ is itself isomorphic to $\GL_n(\mathbb{Z})\simeq\GL(M)$ by \eqref{eq:aut-torus}.  
\begin{align} \label{eq:claim}
    \mbox{We claim that }\varphi\left(-\id\in \GL(M)\right)=-\id\in \GL(M_Y)\,.
\end{align}

Hence, $-\id\in G\subset \GL(M_Y)$ and so the semigroup $S$ is a group since every element has an inverse. Finally, since the action of $\varphi(T)$ on $Y$ is faithful, we have that the group generated by $S$ equals $M_Y$ and so $S=M_Y \simeq \ZZ^{n}$ which proves the proposition.

\medskip

To prove the claim in \eqref{eq:claim}, we
fix the natural embedding
$$
\psi\colon  \GL(M_Y) \subset \GL(M_Y \otimes_{\ZZ} \KK)
$$
and consider the injective homomorphism
$$
\psi\circ \varphi\colon \GL(M) \simeq \GL_n(\ZZ) \to \GL(M_Y \otimes_{\ZZ} \KK)\,.
$$
By Margulis super-rigidity theorem \cite{M91}, see also \cite[Theorem~D]{PSA}, there is a finite index subgroup $K$ of $\SL(M)$, which is a subgroup of index two in $\GL(M)$ and a rational representation 
\[
F\colon \SL(M \otimes_{\ZZ}\KK) \to \GL(M_Y \otimes_{\ZZ} \KK) \quad \mbox{with}\quad F|_K = \psi\circ \varphi  |_K\,.
\]
This implies that the image of $\SL(M)$ by $\psi\circ \varphi$ is dense in $\SL(M_Y \otimes_{\ZZ} \KK)$. 
Therefore, $\psi\circ \varphi$
maps the center of $\GL(M)$, which is $\{ \id_M,-\id_M \}\subset \GL(M)$, into the subgroup of $\GL(M_Y \otimes_{\ZZ} \KK)$ which commutes with $\SL(M_Y\otimes_{\ZZ} \KK)$. This last group is the center of $\GL(M_Y \otimes_{\ZZ} \KK)$ which we denote by $Z(\GL(M_Y \otimes_{\ZZ} \KK))$. 
Hence, $\varphi$ maps the center of $\GL(M)$ to
\[
Z(\GL(M_Y \otimes_{\ZZ} \KK)) \cap \GL(M_Y) = \{  \id_{M_Y},-\id_{M_Y} \}.
\]
Therefore, $\varphi$ maps the center of 
 $\GL(M)$ to the
center of  $\GL(M_Y)$ which proves the claim since $\varphi$ is injective.   
\end{proof}

We are ready to prove the main result of this paper.

\begin{proof}[Proof of \cref{maintheorem}]

Statement \eqref{main1} follows directly from \cref{main-torus} and Statement \eqref{main3} follows directly from \cref{isomofaut}. To prove Statement \eqref{main2} fix an isomorphism
 \[
\varphi\colon \Aut(X) \to \Aut(Y).
\]
Since $X$ is a normal affine toric variety, $X = \KK[\sigma^\vee\cap M]$ for some polyhedral cone $\sigma\subset M_\QQ$. Assume $T\subset \Aut(X)$ is a maximal subtorus of dimension $\dim X$.

By \cref{RvS23TheoremC1},  the image of $T$ inside $\Aut(Y)$ is a maximal subtorus of dimension $\dim T$ and the variety $Y$ is an $n$-dimensional affine toric variety. Hence, $Y = \spec \KK[S]$ for some affine semigroup $S\subset M$. 
By \cref{RvS23Theorem9.1}, 
there exists a field automorphism $\tau$ of
$\KK$ such that the isomorphism
$$\nu\colon M_T\to M_{\varphi(T)}\quad\mbox{given by}\quad \chi\mapsto \chi\circ \left(\varphi^{-1}\right)|_{\varphi(T)}\circ \tau_{\varphi(T)}$$
maps $\mathcal{R}_T(X)$ to $\mathcal{R}_{\varphi(T)}(Y)$. Moreover, by \cref{Th1.7-toric}, if we identify $T$ and $\varphi(T)$ via the isomorphism induced by $\nu$, we have that $\mathcal{R}(\sigma^\vee\cap M) = \mathcal{R}(S)$. By \cref{normalization}, we conclude that the normalization of $Y$ is isomorphic to $X$. Finally, by \cref{A1timesZisdeterminedbyroots}, we conclude $Y\simeq X$ which proves part \eqref{main2} and thus concludes the proof of \cref{maintheorem}.
\end{proof}

\bibliographystyle{alpha}
\bibliography{ref}

\end{document}